\documentclass[11pt]{amsart}

\usepackage{amssymb}
\usepackage{graphicx}  

\def\R{{\mathbb {R}}}

\def\F{{\mathcal{F}}}
\def\J{{\mathcal{J}}}

\def\PP{{\mathcal{P}}}

\def\sop{\operatorname {\mathrm{spt}}}
\def\dist{\operatorname {\mathrm{dist}}}

\def\div{\operatorname {\mathrm{div}}}

\newtheorem{teo}{Theorem}[section]
\newtheorem{lema}[teo]{Lemma}
\newtheorem{prop}[teo]{Proposition}

\theoremstyle{remark}
\newtheorem{remark}[teo]{Remark}

\theoremstyle{definition}

\numberwithin{equation}{section}

\begin{document}

\title[Shape derivative]{A shape-derivative approach to some PDE model in image restoration}

\author[C. Baroncini, J. Fern\'andez Bonder]{Carla Baroncini and Juli\'an Fern\'andez Bonder}

\address{IMAS - CONICET and Departamento de Matem\'atica, FCEyN - Universidad de Buenos Aires, Ciudad Universitaria, Pabell\'on I  (1428) Buenos Aires, Argentina.}

\email[J. Fernandez Bonder]{jfbonder@dm.uba.ar}

\urladdr[J. Fernandez Bonder]{http://mate.dm.uba.ar/~jfbonder}

\email[C. Baroncini]{cbaroncin@dm.uba.ar}

\subjclass[2010]{49Q10,49J45}

\keywords{Shape optimization, sensitivity analysis, nonstandard growth}

\begin{abstract}
In this paper we analyze the shape derivative of a cost functional appearing in image restoration. The main feature of this cost functional is the appearance of a variable exponent. 
\end{abstract}

\maketitle

\section{Introduction}

Shape derivative (or Hadamard derivative) has been proved to be a valuable tool in order to study shape optimization problems. The main ideas go back to Hadamard's original paper \cite{Hadamard} and has been further developed since. See for instance the books \cite{Banichuk, Pironneau, Sokolowski-Zolesio}.

In this paper we are devoted to the analysis of shape derivative of certain functionals arising in image restoration, whose main feature is that it involves a variable exponent.

Let us begin by discussing the model where these functionals appear.

The goal in image restoration is to obtain an image which is modeled by a function $u\colon \Omega \rightarrow \R$, where $\Omega=(0,1)\times(0,1)\subset \R^2$, given that one has a distorted image $I\colon \Omega \rightarrow \R$.

It is customary to assume that the introduced error, $e=u-I$, is small and the objective is to recover $u$ from $I$ without making any further assumptions on the error $e$. 

A classical PDE model introduced by Chambolle and Lions in \cite{Chambolle-Lions} in 1997, propose to obtain $u$ as a minimizer of the functional
$$
\min \frac{1}{2\beta} \left(\int_{\{|\nabla v|\le \beta\}} |\nabla v|^2\, dx + \int_{\{|\nabla v|>\beta\}} |\nabla v|\, dx\right) + \frac{\beta}{2} \int_{\Omega} (v-I)^2\, dx,
$$
where $\beta>0$ is a parameter that needs to be adjusted by the operator of the method for each image.
The idea behind this method is that the real image must be smooth in regions where there are no boundaries (which are interpreted as regions where the derivatives are not big) and, in the ones which contains boundaries, the solution must admit discontinuities.
This method can be re-written as follows
$$
\min \frac{1}{2\beta} \int_\Omega |\nabla v|^{p(|\nabla v|)}\, dx + \frac{\beta}{2}\int_\Omega (v-I)^2\, dx,
$$
where the exponent $p$ is defined as
$$
p(t) = \begin{cases}
2 & \text{if } t\le \beta\\
1 & \text{if } t>\beta.
\end{cases}
$$
This method is extremely difficult to study rigorously since the space where the functional is defined is not a good functional space. That is why, in 2006, Chen, Levine and Rao introduced in \cite{CLR} a modification by which the exponent $p$ is computed from $I$ but it is fixed. In this second model,
$$
p(x) = 1 + \frac{1}{1+k|\nabla G_\sigma * I|^2},
$$
where $G_\sigma(x) = \frac{1}{\sigma} \exp(-|x|^2/4\sigma^2)$ is the Gaussian filter, with $k, \sigma>0$ parameters. Therefore, $p\sim 1$ where $I$ is discontinuous and $p\sim 2$ where $I$ is smooth.

Then, the problem to be minimized is
$$
\min \frac{1}{2\beta} \int_\Omega |\nabla v|^{p(x)}\, dx + \frac{\beta}{2}\int_\Omega (v-I)^2\, dx.
$$
By considering a fixed regular exponent, the authors can use the Sobolev and Lebesgue spaces with variable exponent, thoroughly studied since the sixties. See \cite{Diening}.

Here we consider a variant of these methods, that can be thought of being in between these two,  that approximates the one created by Chambolle and Lions preserving the good functional properties given by the one presented by Chen, Levine and Rao.

We start by dividing the region $\Omega$ into two sub regions $D_1$ and $D_2$ such that for $i=1,2$, 
\begin{equation}\label{prop.particion}
D_i \subset \Omega \text{ is open},\ \mathring{\overline{D_i}} = D_i,\  D_1\cap D_2 = \emptyset,\  \text{and}\ \overline{\Omega} = \overline{D_1}\cup \overline{D_2}.
\end{equation}

By this partition, we make sure that $D_1$ contains the regions with boundaries of the image and $D_2$ its complement. One way of creating this partition is the following:
$$
D_1 = \{x\in\Omega \colon |\nabla G_\sigma * I|> \beta\},\quad D_2 =\{x\in \Omega \colon |\nabla G_\sigma * I|<\beta\}.
$$

We define an exponent  $p\colon \Omega\to \R$ given by
$$
p(x) = \begin{cases}
1+\epsilon & \text{if } x\in D_1\\
2 & \text{if } x\in D_2.
\end{cases}
$$
Then we compute $u$ by minimizing the functional
$$
J(v) = \frac{1}{2\beta}\int_\Omega |\nabla v|^{p(x)}\, dx + \frac{\beta}{2} \int_\Omega (v-I)^2\, dx.
$$
In order to improve the image found, we then may apply an iterative \textit{steepest descent type method} by following the shape derivative of the functional. 

So the main objective of this paper is to compute this shape derivative.

Let us recall that  a related minimization problem was studied in \cite{Acerbi-Fusco}. In that article it is shown that minimizers are H\"older-continuous across the interfase.

Hence we are left with the problem of computing the shape derivative of $J(u)$ with respect to $D_i$, which we describe now.
Given $V\colon \R^N \rightarrow \R^N$ a Lipschitz deformation field, the associated flow $\{\Phi_t\}_{t\in\R}$ is defined by 
\begin{equation}\label{flux}
\begin{cases}
\frac{d}{dt}\Phi_{t}(x) = V(\Phi_{t}(x)),& t\in\R,\ x\in \R^N\\
\Phi_{0}(x) = x & x\in\R^N.
\end{cases}
\end{equation}
Let us observe that $\Phi_t\colon \R^N \to \R^N$ is a group of diffeomorfisms. That is, $\Phi_t\circ \Phi_s = \Phi_{t+s}$ and $\Phi_t^{-1} = \Phi_{-t}$.

We will assume that $\sop(V)\subset \Omega$, so that $\Phi_t(\Omega) = \Omega$ for every $t\in \R$.

Then, the regions $D_i$ are deformed by $\Phi_{t}$ and we obtain a family of partitions $D_i^t = \Phi_t(D_i),\quad i=1,2$ that verify \eqref{prop.particion} and we define
$$
p_t(x) = \begin{cases}
1+\epsilon & \text{if } x\in D_1^t\\
2 & \text{if } x\in D_2^t.
\end{cases}
$$
Observe that $p_t = p\circ \Phi_{-t}$.

Then, for each $t\in\R$ we define the following functional
$$
J_t(v) = \frac{1}{2\beta}\int_\Omega |\nabla v|^{p_t(x)}\, dx + \frac{\beta}{2} \int_\Omega (v-I)^2\, dx,
$$
Let $u_t$ be the minimizer of $J_t$. We can consider the function $j\colon \R\to\R$ given by $j(t) = J_t(u_t)$.

The shape derivative consists then in computing $j'(0)$.

Then, by finding a good expression for such derivative, it will be possible to compute the deformations field $V$ which makes it as negative as possible and so choose the optimal deformation field to then iterate 
$$
D_i^{\Delta t} \simeq (id + \Delta t V)(D_i).
$$

\section{Preliminaries}

Because of the nature of our problem, which deals with piecewise constant exponents, we are unable to assume any regularity on the variable exponent $p$. Therefore, since most of the known results for variable exponent Sobolev spaces assume that the exponent is at least log-H\"older continuous, we need to review the results that are needed here and prove the missing parts in the case of piecewise constant exponents.

\subsection{Definitions and well-known results}

Given $\Omega\subset \R^N$ a bounded open set, we consider the class of exponents $\PP(\Omega)$ given by
$$
\PP(\Omega) := \{p\colon \Omega\to [1,\infty)\colon p \text{ is measurable and bounded}\}.
$$

The variable exponent Lebesgue space $L^{p(x)}(\Omega)$ is defined by
$$
L^{p(x)}(\Omega):= \Big\{f\in L^1_{\text{loc}}(\Omega)\colon \rho_{p(x)}(f)<\infty\Big\},
$$
where the modular $\rho_{p(x)}$ is given by
$$
\rho_{p(x)}(f) := \int_{\Omega} |f|^{p(x)}\, dx.
$$
This space is endowed with the Luxemburg norm
$$
\|f\|_{L^{p(x)}(\Omega)} = \|f\|_{p(x),\Omega} = \|f\|_{p(x)} := \sup\Big\{\lambda>0\colon \rho_{p(x)}(\tfrac{f}{\lambda})<1\Big\}. 
$$

The infimum and the supremum of the exponent $p$ play an important role in the estimates as the next elementary proposition shows. For further references, the following notation will be imposed
$$
1\le p_-:= \inf_{\Omega}p \le \sup_{\Omega} p =: p_+<\infty.
$$
The proof of the following proposition can be found in \cite[Theorem 1.3, p.p. 427]{FanyZhao}.

\begin{prop}\label{propdesigualdades}
Let $f\in L^{p(x)}(\Omega)$, then
$$
\min\{\|f\|_{p(x)}^{p_-}, \|f\|_{p(x)}^{p_+}\} \leq \rho_{p(x)}(f)\leq \max\{\|f\|_{p(x)}^{p_-}, \|f\|_{p(x)}^{p_{+}}\}.
$$
\end{prop}

\begin{remark}\label{minmax}
Proposition \ref{propdesigualdades}, is equivalent to
$$
\min\{\rho_{p(x)}(f)^{\frac{1}{p_-}}, \rho_{p(x)}(f)^{\frac{1}{p_{+}}} \} \leq \|f\|_{p(x)}\leq \max\{\rho_{p(x)}(f)^{\frac{1}{p_-}}, \rho_{p(x)}(f)^{\frac{1}{p_{+}}} \}.
$$
\end{remark}

We will use the following form of H\"older's inequality for variable exponents. The proof, which is an easy consequence of Young's inequality, can be found in \cite[Lemma 3.2.20]{Diening}.

\begin{prop}[H\"older's inequality]\label{propholder}
Assume $p_->1$. Let $u\in L^{p(x)}(\Omega)$ and $v\in L^{p'(x)}(\Omega)$, then
$$
\int_{\Omega} |u v|\, dx\leq 2\|u\|_{p(x)} \|v\|_{p'(x)},
$$
where $p'(x)$ is, as usual, the conjugate exponent, i.e. $p'(x):= p(x)/(p(x)-1)$.
\end{prop}

The variable exponent Sobolev space $W^{1,p(x)}$ is defined by
$$
W^{1,p(x)}(\Omega):=\Big\{u\in W^{1,1}_{\text{loc}}(\Omega)\colon u\in L^{p(x)}(\Omega) \text{ and } \partial_i u\in L^{p(x)}(\Omega)\ i=1,\dots,N\Big\},
$$
where $\partial_i u$ stands fot the $i-$th partial weak derivative of $u$.

This space posses a natural modular given by
$$
\rho_{1,p(x)}(u) := \int_\Omega |u|^{p(x)} + |\nabla u|^{p(x)}\, dx,
$$
so $u\in W^{1,p(x)}(\Omega)$ if and only if $\rho_{1,p(x)}(u)<\infty$.

The corresponding Luxemburg norm associated to this modular is
$$
\|u\|_{W^{1,p(x)}(\Omega)} = \|u\|_{1,p(x),\Omega} = \|u\|_{1,p(x)} := \sup\Big\{\lambda>0\colon \rho_{1,p(x)}(\tfrac{u}{\lambda})<1\Big\}. 
$$
Observe that this norm turns out to be equivalent to $\|u\|:= \|u\|_{p(x)} + \|\nabla u\|_{p(x)}$.

Now we state and prove a simple proposition that characterizes the Sobolev space when the variable exponent is piecewise constant.

\begin{prop}
Let $\Omega\subset \R^N$ be an open of finite measure and let $D_1, D_2\subset \Omega$ be a partition verifying \eqref{prop.particion}. Let $1\le p_1,p_2<\infty$ and let $p\in \PP(\Omega)$ be such that $p=p_1\chi_{D_1} + p_2\chi_{D_2}$.

Then, $u \in W^{1,p(x)}(\Omega)\Leftrightarrow u \in W^{1,p_1}(D_1),\  u \in W^{1,p_2}(D_2) \text{ and u} \in W^{1,\min \{p_1,p_2\}}(\Omega).$
\end{prop}

\begin{proof}
Observe that \eqref{prop.particion} implies that $|\Omega \setminus (D_1\cup D_2)|=0$. Then
$$
\int_{\Omega} |u|^{p(x)}\, dx = \int_{D_1} |u|^{p_1}\, dx + \int_{D_2} |u|^{p_2}\, dx.
$$
and
$$
\int_{\Omega} |\nabla u|^{p(x)}\, dx = \int_{D_1} |\nabla u|^{p_1}\, dx + \int_{D_2} |\nabla u|^{p_2}\, dx.
$$

Moreover, assume that $p_1<p_2$ and by H\"older's inequality,
\begin{align*}
\int_{\Omega} |\nabla u|^{p_1}\, dx &=  \int_{D_1} |\nabla u|^{p_1}\, dx + \int_{D_2} |\nabla u|^{p_1}\, dx\\
&\le \int_{D_1} |\nabla u|^{p_1}\, dx + |D_2|^{\frac{p_2-p_1}{p_2}}\left(\int_{D_2} |\nabla u|^{p_2}\, dx\right)^{\frac{p_1}{p_2}}<\infty.
\end{align*}
Analogously, $u\in L^{p_1}(\Omega)$.

For the converse, we just observe that since $u\in W^{1, \min\{p_1,p_2\}}(\Omega)$, then $\nabla u$ is defined in the whole of $\Omega$. Then is easy to see that $\nabla u\in L^{p(x)}(\Omega)$ by the same arguments as before.
\end{proof}

\section{Differentiability.}

Let $V$ be a Lipschitz vector field with support in $\Omega$ and let $\{\Phi_t\}_{t\in\R}$ its associated flux given by  \eqref{flux}.

Let us begin with the following observation:
\begin{remark}\label{asintoticas} 
By Taylor expansion, we have
$$
\Phi_t(x)=x+V(x)t+o(t)
$$
and so we have the following asymptotic formulas hold:
$$
D\Phi_t(x)=Id+tDV(x)+o(t)=Id+O(t),
$$
$$
J\Phi_t(x)=1+t \div V (x)+o(t)=1+O(t),
$$
for all $x \in \R^N$, where $J\Phi_t$ is the Jacobian of $\Phi_t$.
\end{remark}

The following proposition, though elementary, will be useful in the sequel and shows that any diffeomorphism $\Phi\colon \R^N\to\R^N$, induces a bounded linear isomorphism between Sobolev spaces.

\begin{prop}\label{prop.iso}
Let $\Phi \colon \Omega_1 \to \Omega_2$ be a diffeomorphism and $p \in \PP(\Omega_1)$ be a bounded exponent. 

Then, $\Phi$ induces a bounded linear isomorphism
$$
\F \colon W^{1,p}(\Omega_1) \to W^{1,q}(\Omega_2),
$$
where $q \colon \Omega_2 \to [1, +\infty)$ is given by $q(x):=p(\Phi^{-1}(x))$, by the expression
$$
\F(u):=u\circ \Phi^{-1}.
$$
\end{prop}

\begin{proof}	
We first observe that $\F$ is clearly a linear isomorphism with inverse given by 
$$
\F^{-1} \colon W^{1,q}(\Omega_2) \to W^{1,p}(\Omega_1), \qquad \F^{-1}(v):=v\circ \Phi.
$$
Let us now see that it is also bounded. 

Let us consider $\lambda>0$ and, for simplicity, let us denote $v=\F(u)$. Then, by changing variables $y=\Phi^{-1}(x)$,
\begin{align*}
\int_{\Omega_2}\Big|\frac{v(x)}{\lambda}\Big|^{q(x)}\, dx&=\int_{\Omega_2}\Big|\frac{u(\Phi^{-1}(x))}{\lambda}\Big|^{p(\Phi^{-1}(x))}\, dx\\
&=\int_{\Omega_1}\Big|\frac{u(y)}{\lambda}\Big|^{p(y)}J\Phi(y)\, dy\\
&\leq \|J \Phi\|_{\infty}\int_{\Omega_1}\Big|\frac{u(y)}{\lambda}\Big|^{p(y)}dy
\end{align*}
Let us observe that, if $C:=\|J \Phi\|_{\infty}\leq 1$, clearly we have 
$$
\|u\|_{p,\Omega_1}=\inf \{\lambda>0 \colon \int_{\Omega_1}\Big|\frac{u(y)}{\lambda}\Big|^{p(y)}dy\leq 1\}\geq \inf \{\lambda>0 \colon \int_{\Omega_2}\Big|\frac{v(y)}{\lambda}\Big|^{q(y)}dy\leq 1\}=\|v\|_{q,\Omega_2}
$$
Let us now assume that $C>1$. Then, since
\begin{align*}
\left\{\lambda>0 \colon \int_{\Omega_1}\Big|\frac{C^{\frac{1}{p_-}}u(y)}{\lambda} \Big|^{p(y)}\, dy \leq 1\right\} & \subset \left\{\lambda>0 \colon \int_{\Omega_1}\Big|\frac{u(y)}{\lambda} \Big|^{p(y)}\, dy \leq \frac{1}{C}\right\}\\
&\subset \left\{\lambda>0 \colon \int_{\Omega_2}\Big|\frac{v(x)}{\lambda} \Big|^{q(x)}\, dx\leq 1\right\},
\end{align*}
taking infimum, we conclude that
$$
C^{\frac{1}{p_-}}\|u\|_{p,\Omega_1}=\|C^{\frac{1}{p_-}}u\|_{p,\Omega_1}\geq\inf\{\lambda>0 \colon \int_{\Omega_1}\Big|\frac{u(y)}{\lambda} \Big|^{p(y)}\, dy \leq \frac{1}{C}\}\geq \|v\|_{q,\Omega_2}=\|\F(u)\|_{q, \Omega_2}.
$$
Analogously,
\begin{align*}
\int_{\Omega_2}\Big|\frac{\nabla v(x)}{\lambda}\Big|^{q(x)}\, dx &= \int_{\Omega_2}\Big|\frac{\nabla (u\circ \Phi^{-1})(x)}{\lambda}\Big|^{q(x)}\, dx\\
&= \int_{\Omega_2}\Big|\frac{\nabla u(\Phi^{-1}(x)) D\Phi^{-1}(x)}{\lambda}\Big|^{p(\Phi^{-1}(x))}\, dx\\
&= \int_{\Omega_1}\Big|\frac{\nabla u(y) D\Phi^{-1}(\Phi(y))}{\lambda}\Big|^{p(y)}J \Phi(y)\, dy\\
&\leq \max\{1, \|D \Phi^{-1}\|_{\infty}\}^{p_+} \|J \Phi\|_{\infty}\int_{\Omega_1}\Big|\frac{\nabla u(y)}{\lambda}\Big|^{p(y)}\, dy.
\end{align*}

Therefore, $\|\nabla \F(u)\|_{q, \Omega_2}\leq C \|\nabla u\|_{p, \Omega_1}$, which completes the proof.
\end{proof}

\begin{remark}
In the previous proof, given $A \colon \Omega \to \mathbb{R}^{N\times N}$, we considered the norm $\|A\|_{\infty}:=\displaystyle{\sup_{x\in \Omega}}\|A(x)\|$ and, given $B \in \mathbb{R}^{N\times N}$, we considered the norm $\|B\|:=\displaystyle{\sup_{\xi\neq 0}}\frac{|B\xi|}{|\xi|}$.
\end{remark}

Observe that, since $\sop(V)\subset\subset \Omega$, it follows that $\Phi_t(\Omega)=\Omega$ for every $t\in\R$ and that if $p = p_1\chi_{D_1} + p_2\chi_{D_2}$ then $p_t:=p\circ \Phi_{-t} = p_1 \chi_{D_1^t} + p_2 \chi_{D_2^t}$, where $D_i^t = \Phi_t(D_i)$, $i=1,2$.

Therefore, in view of Proposition \ref{prop.iso}, we have that 
$$
\F_t\colon W^{1,p}(\Omega)\to W^{1,p_t}(\Omega), \qquad u\mapsto u\circ \Phi_{-t}
$$
is a bounded linear isomorphism.

Let us consider the space $X_t:=W^{1,p_t}(\Omega)\cap L^{2}(\Omega)$ equipped with the norm 
$$
\|\cdot\|_{X_t}:=\|\cdot\|_{W^{1,p_t}(\Omega)}+\|\cdot\|_{L^{2}(\Omega)}
$$
and the space $X:=W^{1,p}(\Omega)\cap L^{2}(\Omega)$ equipped with the norm 
$$
\|\cdot\|_X:=\|\cdot\|_{W^{1,p}(\Omega)}+\|\cdot\|_{L^{2}(\Omega)}.
$$
It is clear that $\F_t\colon X\to X_t$ is still a bounded linear isomorphism.
 
Given $f \in L^{2}(\Omega)$, we define the quantity
$$
\tilde s(t):= \inf_{v \in X_t} \int_{\Omega}\frac{|\nabla v|^{p_t}}{p_t}\, dx +\int_{\Omega}\frac{|v-f|^2}{2}\, dx
$$
which is clearly equivalent to 
\begin{equation}\label{st}
s(t):= \inf_{v \in X_t} \int_{\Omega}\frac{|\nabla v|^{p_t}}{p_t}\, dx +\int_{\Omega} \frac{|v|^2}{2} \, dx- \int_{\Omega}v f\, dx.
\end{equation}
In fact, $\tilde s(t) = s(t) + \|f\|_2^2$.

Observe that, since $\F_t$ is an isomorphism, one actually has
$$
s(t) = \inf_{u\in X} \int_{\Omega} \frac{|\nabla (u\circ \Phi_{-t})|^{p_t}}{p_t}\, dx + \int_{\Omega} \frac{|u\circ \Phi_{-t}|^2}{2}\, dx - \int_{\Omega} (u\circ \Phi_{-t}) f\, dx.
$$

So, in view ot our previous discussions, our primary goal is to find an expression for $\frac{ds}{dt}(0)$.

\begin{remark}\label{rem.equivalencia}
Let us observe that, by changing variables $y=\Phi_{-t}(x)$,
$$
s(t)= \inf_{u \in X} \int_{\Omega}\frac{|\nabla u D\Phi_{-t}\circ \Phi_{t} |^{p}}{p} J\Phi_t\, dy +\int_{\Omega} \frac{|u|^2}{2} J\Phi_t\, dy -\int_{\Omega} u f\circ \Phi_t J\Phi_t\, dy.
$$
\end{remark}

Let us call 
$$
\J_t u := \int_{\Omega}\frac{|\nabla u D\Phi_{-t}\circ \Phi_{t} |^{p}}{p} J\Phi_t\, dy +\int_{\Omega} \frac{|u|^2}{2} J\Phi_t\, dy - \int_{\Omega} u f\circ \Phi_t J\Phi_t\, dy
$$
and
$$
\J u := \int_{\Omega}\frac{|\nabla u|^{p}}{p}\, dy + \int_{\Omega} \frac{|u|^2}{2} \, dy - \int_{\Omega}u f\, dy.
$$

\begin{lema}\label{equicoercividad}
There exists $\delta>0$ such that the functionals $\{\J_t\}_{|t|<\delta}$ are uniformly coercive with respect to the weak topology of $X$. That is, for any $\lambda\in\R$, there exists a weakly compact set $K\subset X$ such that
$$
\{\J_t \le \lambda\} \subset K, \quad \text{for every } |t|<\delta.
$$ 
\end{lema}

\begin{proof}
Take $\delta>0$ such that $\frac{1}{2}\leq J\Phi_t \leq 2$. Therefore, 
\begin{equation}\label{cotaJ}
\J_t u \geq \frac{1}{2}\int_{\Omega}\frac{|\nabla u D\Phi_{-t}\circ \Phi_{t} |^{p}}{p} \, dy +\frac{1}{2}\int_{\Omega} \frac{|u|^2}{2} \, dy - 2 \int_{\Omega} |f| |u|\, dy.
\end{equation}
By Young inequality with $\epsilon=\frac{1}{8}$,
\begin{equation}\label{young4}
2\int_{\Omega} |f| |u|\, dy \leq \frac{1}{8}\int_{\Omega} |u|^2\, dy+8\int_\Omega |f|^2\, dy.
\end{equation}

As $D\Phi_{-t}\rightrightarrows Id$ uniformly on $\Omega$, it follows that $\|D\Phi_t\|_\infty$ is bounded away from zero and infinity for every $|t|< \delta$, so
\begin{equation}\label{D-Id}
\int_{\Omega}\frac{|\nabla u D\Phi_{-t}\circ \Phi_{t} |^{p}}{p} \, dy\geq c \int_{\Omega}|\nabla u|^{p} \, dy.
\end{equation}
So, combining \eqref{cotaJ}, \eqref{young4} and \eqref{D-Id}, we get
$$
\J_t u \geq c\int_\Omega |\nabla u|^p\, dy + \frac18 \int_\Omega |u|^2\, dy - 8\|f\|_2^2.
$$
By Proposition \ref{minmax} we easily conclude that there exists a radius $R=R(\lambda)$ such that $\{\J_t\le \lambda\}\subset B_X(0, R)$.

Therefore, if we denote $K:=\{\|u\|_{X}<R\}$, satisfies our requirements. This finishes the proof of the lemma.
\end{proof}

The next lemma is stated for future reference, its proof is standard.
\begin{lema}\label{lemmaextremales}
There exists a unique extremal for $s(t)$ and $s(0)$.
\end{lema}

\begin{proof}
The proof is an immediate consequence of the fact that both $\J_t$ and $\J$ are strictly convex and sequentially weakly lower semicontinuous on $W^{1,p}(\Omega)$.
\end{proof}

Our first result shows that $s(t)$ is continuous with respect to $t$ at $t=0$.

\begin{teo}\label{teo.continuidad}
With the previous notation, 
\begin{equation}\label{st.to.s}
\lim_{t \stackrel{}{\rightarrow}0^+}s(t)=s(0).
\end{equation}
Moreover, if $u_t$ and $u$ are the extremals associated to $s(t)$ and $s(0)$ respectively, then $u_t\rightharpoonup u$ weakly in $W^{1,p}(\Omega)$. Finally, if $p^*:=\frac{pN}{N-p}>2$ then $u_t\to u$ strongly in $W^{1,p}(\Omega)$.
\end{teo}

\begin{remark}
The hypothesis $p^*>2$ is needed in order to secure the compact embedding $W^{1,p}(\Omega)\subset L^{2}(\Omega)$ for any dimension $N$.

For the case $N=2$, one has $p^*>2$ for any $p>1$ so no extra hypothesis is needed.
\end{remark}

\begin{proof}
Since, by Lemma \ref{equicoercividad}, we know that the functionals $\J_t$ are uniformly coercive, the proof of \eqref{st.to.s} will follow from Remark \ref{obs gamma convergencia} if we show that $\J_t \rightrightarrows \J$ uniformly on bounded sets of $X$. Observe that since the minimizers are unique, we will then have that the whole sequence of minimizers is weakly convergent.

Let us consider now $B \subset X$ a bounded subset and $u \in B$. By Remark \ref{asintoticas},
\begin{align*}
\J_t u=&\int_{\Omega}\frac{|\nabla u (Id+O(t)) |^{p}}{p} (1+O(t))\, dy +\int_{\Omega} \frac{|u|^2}{2} (1+O(t))\, dy - \int_{\Omega} u (f\circ \Phi_{t})(1+O(t))\, dy\\
=&(1+O(t))\left\{\int_{\Omega}\frac{|\nabla u(Id +O(t))|^{p}}{p} \, dy +\int_{\Omega} \frac{|u|^2}{2}\, dy - \int_{\Omega} u (f\circ \Phi_t)\, dy \right\}.
\end{align*}
Again by Remark \ref{asintoticas}, and by Taylor expansion formula, we get
$$
\int_{\Omega}\frac{|\nabla u(Id +O(t))|^{p}}{p} \, dy=\int_{\Omega}\frac{|\nabla u|^{p}}{p}\, dy + O(t),
$$
uniformly in $B$.

Assume for a moment that $f$ is a continuous function with compact support. Then, since $\Phi_{t}\to id$ uniformly as $t \to 0$, we have that $f\circ\Phi_t\to f$ uniformly as $t \to 0$ and therefore,
$$
\|f\circ\Phi_t-f\|^{2}_{2}=\int_{\Omega}|f\circ\Phi_t-f|^{2}\, dx\leq \|f\circ\Phi_t-f\|^{2}_{\infty}|\Omega|\to 0,\ (t\to 0).
$$ 
And so we have that $\|f\circ\Phi_t-f\|_{2}\to 0$,$\ (t\to 0)$.

Now, by a standard density argument, it is easy to see that the same result holds for any $f \in L^{2}(\Omega)$.


Then, by H\"older inequality and since $u \in B$, there is a constant $C$, independent of $u$, such that
$$
\Big|\int_{\Omega}u(f\circ \Phi_t-f)\Big|\leq C\|f\circ \Phi_t - f\|_2\to 0
$$
as $t\to 0^+$.

Assume now that $p^*>2$. It remains to see the strong convergence of $u_t$ to $u$ in $W^{1,p}(\Omega)$.

Let us observe that in order to see the strong convergence it is enough to show the convergence of the modulars (see \cite{Diening}).

Let us now recall that
\begin{align*}
\int_{\Omega}\frac{|\nabla u_t|^{p}}{p}\, dy+\int_{\Omega} \frac{|u_t|^2}{2}\, dy =& s(t) + \int_{\Omega}\frac{|\nabla u_t|^{p}}{p}\, dy-\int_{\Omega}\frac{|\nabla u_t D\Phi_{-t}\circ \Phi_{t} |^{p}}{p} J\Phi_t\, dy\\
&+\int_{\Omega} \frac{|u_t|^2}{2}(1-J\Phi_t)\, dy+\int_{\Omega} u_t (f\circ \Phi_t) J\Phi_t\, dy.
\end{align*}
By Remark \ref{asintoticas},
\begin{align*}
\int_{\Omega}\frac{|\nabla u_t D\Phi_{-t}\circ \Phi_{t} |^{p}}{p} J\Phi_t\, dy &=\int_{\Omega} \frac{|\nabla u_t - t\nabla u_t DV + o(t)|^{p}}{p} (1 + t\div V + o(t))\, dy. 
\end{align*} 
Using the following Taylor expansion, 
$$
|\nabla u_t - t \nabla u_t DV + o(t)|^{p} = |\nabla u_t|^{p}- p t |\nabla u_t|^{p-2} \nabla u_t\cdot\nabla u_t DV + o(t),
$$
we find that
$$
\int_{\Omega}\frac{|\nabla u_t D\Phi_{-t}\circ \Phi_{t}|^{p}}{p} J\Phi_t\, dy = \int_{\Omega}\frac{|\nabla u_t|^{p}+t(|\nabla u_t|^{p}\div V-p|\nabla u_t|^{p-2}\nabla u_t\cdot\nabla u_t DV)}{p}\, dy+o(t). 
$$
And so we have
$$
\int_{\Omega}\frac{|\nabla u_t|^{p}}{p}\, dy-\int_{\Omega}\frac{|\nabla u_t D\Phi_{-t}\circ \Phi_{t} |^{p}}{p} J\Phi_t\, dy=-\int_{\Omega}\frac{t(|\nabla u_t|^{p}\div V-p|\nabla u_t|^{p-2}\nabla u_t\cdot\nabla u_t DV)}{p}\, dy+o(t)
$$

Now, for our fourth term, we only need to observe that $\frac{|u_t|^2}{2}$ is bounded and $1-J\Phi_t \rightarrow 0$ uniformly.

Then, since $s(t)\rightarrow s(0)$ and 
$$
\int_{\Omega} u_t (f\circ \Phi_t) J\Phi_t\, dy\rightarrow \int_{\Omega} u f,
$$
we can conclude that
$$
\int_{\Omega}\frac{|\nabla u_t|^{p}}{p}\, dy+\int_{\Omega} \frac{|u_t|^2}{2}\, dy \rightarrow \int_{\Omega}\frac{|\nabla u|^{p}}{p}\, dy+\int_{\Omega} \frac{|u|^2}{2}\, dy,
$$
which completes the proof.
\end{proof}

Now we prove the main result of the section, namely the differentiability of the cost functional $s(t)$. For this result we will need the function $f$ to be of class $C^1$. 
\begin{teo}
$s(t)$ is differentiable at $t=0$ and
$$
\frac{ds}{dt}(0)=R(u) - \int_\Omega uf \div V\, dy - \int_\Omega u \nabla f\cdot V\, dy,
$$
where
$$
R(u):=\int_{\Omega}\frac{|\nabla u|^{p}}{p}\div V-|\nabla u|^{p-2} \nabla u\cdot\nabla uDV+ \div V\frac{|u|^2}{2}\, dy
$$
and $u$ is the extremal of $s(0)$.
\end{teo}

\begin{proof}
By Lemma \ref{lemmaextremales}, we can consider $u$ the extremal of $s(0)$. Then, by Remark \ref{rem.equivalencia},
$$
s(t)=\inf_{X} \J_{t} \leq \J_{t}(u) = \int_{\Omega}\frac{|\nabla u D\Phi_{-t}\circ \Phi_{t} |^{p}}{p} J\Phi_t\, dy +\int_{\Omega} \frac{|u|^2}{2} J\Phi_t\, dy - \int_{\Omega} u f\circ \Phi_t J\Phi_t\, dy.
$$
Now, by Remark \ref{asintoticas}, as in the proof of Theorem \ref{teo.continuidad}
we find that
$$
\int_{\Omega}\frac{|\nabla u D\Phi_{-t}\circ \Phi_{t}|^{p}}{p} J\Phi_t\, dy = \int_{\Omega}\frac{|\nabla u|^{p}+t(|\nabla u|^{p}\div V-p|\nabla u|^{p-2}\nabla u\cdot\nabla uDV)}{p}\, dy+o(t). 
$$
On the other hand, again by Remark \ref{asintoticas},
\begin{align*}
\int_{\Omega} \frac{|u|^2}{2} J\Phi_t\, dy &= \int_{\Omega} \frac{|u|^2}{2} (1+t\div V+o(t))\, dy\\
&= \int_{\Omega} \frac{|u|^2}{2}\, dy+t\int_{\Omega} \div V\frac{|u|^2}{2}\, dy+o(t).
\end{align*}
Therefore, setting
$$
R(u):=\int_{\Omega}\frac{|\nabla u|^{p}}{p}\div V-|\nabla u|^{p-2}\nabla u\cdot\nabla uDV+ \div V\frac{|u|^2}{2}\, dy,
$$
we can conclude that
$$
s(t)\leq \int_{\Omega}\frac{|\nabla u|^{p}}{p}\, dy+\int_{\Omega} \frac{|u|^2}{2}\, dy+tR(u)+o(t)-\int_{\Omega} u (f\circ \Phi_t) (1 + t\div V + o(t))\, dy.
$$
Recall that
$$
s(0) = \int_{\Omega}\frac{|\nabla u|^{p}}{p}\, dy+\int_{\Omega} \frac{u^2}{2}\, dy - \int_\Omega u f\, dy.
$$
Therefore,
$$
\frac{s(t)-s(0)}{t}\leq R(u)+\frac{o(t)}{t}-\int_{\Omega}u (f\circ \Phi_t) \div V\, dy - \int_\Omega u \frac{(f\circ \Phi_t) - f}{t}\, dy.
$$
Taking the limit $t\to 0^{+}$, we get
$$
\limsup_{t\to 0+} \frac{s(t)-s(0)}{t} \leq R(u) - \int_\Omega uf \div V\, dy - \int_\Omega u \nabla f\cdot V\, dy,
$$
where we have used the fact that $\Phi_0=id$ and $\dot \Phi_t = V\circ \Phi_t$.

Let us consider now $\{t_n\}_{n \in \mathbb{N}}$ such that $t_{n}\rightarrow 0^+$ and
$$
\liminf_{t \rightarrow 0^+}\frac{s(t)-s(0)}{t}=\lim_{n \rightarrow \infty}\frac{s(t_n)-s(0)}{t_n}.
$$
Let $u_n:=u_{t_n}\in X_{t_n}$ be the extremal associated to $s(t_n)$. By Remark \ref{rem.equivalencia},
$$
s(t_n)= \int_{\Omega}\frac{|\nabla u_n D\Phi_{-t_n}\circ \Phi_{t_n} |^{p}}{p} J\Phi_{t_n}\, dy +\int_{\Omega} \frac{|u_n|^2}{2} J\Phi_{t_n}\, dy -\int_{\Omega} u_{t_n} f\circ \Phi_{t_n} J\Phi_{t_n}\, dy
$$
Arguing as in the previous case, we have that
\begin{align*}
\frac{s(t_n)-s(0)}{t_n}\geq & \int_{\Omega}\frac{|\nabla u_n|^{p}}{p}\div V-|\nabla u_n|^{p-2}\nabla u_n\cdot\nabla u_{n}DV+ \div V\frac{|u_n|^2}{2}\, dy\\
&+\frac{o(t_n)}{t_n}-\int_{\Omega}u_{n} (f\circ \Phi_{t_n}) \div V\, dy - \int_\Omega u_n \frac{(f\circ \Phi_{t_n}) - f}{t_n}\, dy\\
=& R(u_n) + \frac{o(t_n)}{t_n}-\int_{\Omega}u_{n} (f\circ \Phi_{t_n}) \div V\, dy - \int_\Omega u_n \frac{(f\circ \Phi_{t_n}) - f}{t_n}\, dy.
\end{align*}
Since $R(u_n) \rightarrow R(u)$ when $n \rightarrow \infty$ (just observe that $R$ is continuous with respect to the strong topology and $u_n\to u$ in $W^{1,p}(\Omega)$ by Theorem \ref{teo.continuidad}), we have
$$
\liminf_{t \rightarrow 0^+}\frac{s(t)-s(0)}{t}\geq R(u) - \int_\Omega uf \div V\, dy - \int_\Omega u \nabla f\cdot V\, dy. 
$$
And so we can conclude that $s(t)$ is differentiable at $t=0$ and
$$
\frac{ds}{dt}(0)=R(u) - \int_\Omega uf \div V\, dy - \int_\Omega u \nabla f\cdot V\, dy,
$$
where $u\in X$ is the extremal of $s(0)$. This completes the proof.
\end{proof}

\section{Improvement of the formula.}

Now we try to find a more explicit formula for $s'(0)$. In the following study, we will need the solution $u$ to
\begin{equation}
\left\{
\begin{array}{rl}
-\Delta_{p(x)}u+u=f & \text{in } \Omega,\\
u=0 & \text{on } \partial\Omega,
\end{array} \right.
\end{equation}
to be $C_{\text{loc}}^2(D_1)\cap C^2_{\text{loc}}(D_2)$ in order for our computations to work. However, this is not true since the optimal regularity is known to be $C_{\text{loc}}^{1,\alpha}(D_1)\cap C^{1,\alpha}_{\text{loc}}(D_2)$. See \cite{Tolksdorf}.

In order to overcome such difficulty, we will proceed as follows.

\subsection{Domain regularization.}

Let us first define $D_i(t):=\Phi_t(D_i)$.

Now given a fixed $\delta>0$, we define the following sets
$$
D^{\delta}_i:=\{x\in D_i \colon \dist(x,D_j)>\delta\},\ i\neq j
$$
and consider $D^{\delta}_i(t):=\Phi_t(D^{\delta}_i)$. And now consider the sets
$$
\Gamma^{\delta}_i(t):= \partial D^{\delta}_i(t) \cap \Omega.
$$
Let us observe that, in each $D^{\delta}_i$, the exponent $p(x)=p_i$ is constant so we can apply the classic regularity results. See for instance \cite{Tolksdorf}.

Now we define the sets $A_\delta:= \Omega \setminus (D_1^\delta \cup D_2^\delta)$ and observe that
$$
\partial A_\delta \cap \Omega=\Gamma^{\delta}_1\cup \Gamma^{\delta}_2,
\qquad \Omega=D^{\delta}_1\cup D^{\delta}_2\cup A^{\delta}.
$$

See Figure 1.

\begin{figure}\label{figura1}\caption{Partition of $\Omega$}
\begin{center}
    \includegraphics[width=7cm]{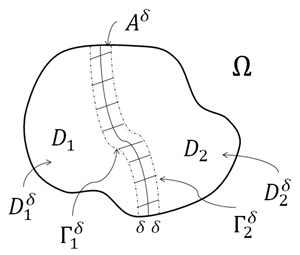}
\end{center}
\end{figure}

\subsection{Operator regularization.}
Now, for $\epsilon \geq 0$, we consider the regularized problems
\begin{equation}\label{eqregul}
\left\{
\begin{array}{rl}
-\div((|\nabla v|^{2}+\epsilon^{2})^{\frac{p_i-2}{2}}\nabla v)+v=f^{\epsilon} & \text{in } D^{\delta}_i(t),\\
v=0 & \text{on } \partial\Omega \cap (\overline{D^{\delta}_1(t)}\cup \overline{D^{\delta}_2(t)}),\\
v=u(0)\circ \Phi^{-1}_t & \text{on } \Gamma^{\delta}_i(t),
\end{array} \right.
\end{equation}
with $f^{\epsilon}\in C^{\infty}$ such that $f^{\epsilon}\rightarrow f$ in $L^{p'}$. 

\begin{remark}
Applying classical estimates (see for instance \cite{GT} it is possible to see that the solution of $\eqref{eqregul}$ is $C^{2,\alpha}_{\text{loc}}(D^{\delta}_i)\cap C^{1}(\overline{D^{\delta}_i})$ if $\epsilon > 0$, since $u(0)$ is $C^1(\overline{D^{\delta}_i})$ and $\Phi_t$ is the identity map in a neighborhood of $\partial\Omega$. See also \cite{Tolksdorf} for regularity estimates in Sobolev spaces.
\end{remark}

Let us define the following sets
$$
X^{\delta}_i:=\{v \in W^{1,p_i}(D^{\delta}_i) \text{ such that } v=0 \text{ in } \partial\Omega \cap \overline{D^{\delta}_i} \text{ and } v=u  \text{ in } \Gamma^{\delta}_i\}.
$$
$$
X^{\delta}_i(t):=\{v \in W^{1,p_i}(D^{\delta}_i(t)) \text{ such that } v=0 \text{ in } \partial\Omega \cap \overline{D^{\delta}_i(t)} \text{ and } v=u(0)\circ \Phi^{-1}_t  \text{ in } \Gamma^{\delta}_i(t)\}.
$$

Let us also consider the functionals $\tilde{\J}^{\epsilon,\delta}_{t,i}\colon X^{\delta}_i(t) \to \R$ defined by
$$
\tilde{\J}^{\epsilon,\delta}_{t,i}(v):=\int_{D^{\delta}_i(t)}\frac{\big(|\nabla v|^2 
 + \epsilon^2\big)^{\frac{p_i}{2}}}{p_i}\, dy +\int_{D^{\delta}_i(t)} \frac{|v|^2}{2}\, dy -\int_{D^{\delta}_i(t)} v f^{\epsilon}\, dy.
$$

\begin{remark}
$X^{\delta}_i(t)$ is strongly closed and convex, therefore it is weakly closed. 
\end{remark}

\begin{remark}
The solutions of \eqref{eqregul} are the minimums of the functionals $\tilde{\J}^{\epsilon,\delta}_{t,i}$ in $X^{\delta}_i(t)$.

Since the functional $\tilde{\J}^{\epsilon,\delta}_{t,i}$ is continuous for the strong topology, strictly convex and coercive, it has a unique minimum in $X^{\delta}_i(t)$ and, therefore, \eqref{eqregul} has a unique weak solution.

We will denote $\tilde{u}^{\epsilon,\delta}_i(t)$ as the function where the minimum is attained. 
\end{remark}

\begin{remark}
Observe that $\psi_t \colon X^{\delta}_i \to X^{\delta}_{i}(t)$ defined by $v \mapsto v \circ \Phi_{t}^{-1}$ is a biyection between $X^{\delta}_i$ and $X^{\delta}_{i}(t)$ and the following equality holds
$$
\J^{\epsilon,\delta}_{t,i}=\tilde{\J}^{\epsilon,\delta}_{t,i}\circ \psi_{t}^{-1}.
$$
\end{remark}

By changing variables as in the previous section we get the functional $\J^{\epsilon,\delta}_{t,i}\colon X^{\delta}_i \to \R$ given by
$$
\J^{\epsilon,\delta}_{t,i}(v):=\int_{D^{\delta}_i}\frac{\big(|\nabla v D\Phi_{-t}\circ \Phi_{t} |^2 
 + \epsilon^2\big)^{\frac{p_i}{2}}}{p_i} J\Phi_t\, dy +\int_{D^{\delta}_i} \frac{|v|^2}{2} J\Phi_t\, dy -\int_{D^{\delta}_i} v f^{\epsilon}\circ \Phi_t J\Phi_t\, dy.
$$
and define
$$
s^{\epsilon,\delta}_i(t)= \inf_{v \in X^{\delta}_i} \J^{\epsilon,\delta}_{t,i}(v) = \inf_{v \in X^{\delta}_i(t)} \tilde \J^{\epsilon,\delta}_{t,i}(v).
$$
We will denote $u^{\epsilon,\delta}_i(t)\in X_i^\delta$ as the function where the minimum of $\J^{\epsilon,\delta}_{t,i}$ is attained. 

Observe that $u^{\epsilon,\delta}_i(t)(x)=\tilde{u}^{\epsilon,\delta}_i(t)(\Phi_t(x))$.

In order to make the notation lighter, we will focus on the needed parameter in each step. First, $u_i$, then $u^{\epsilon}$ and finally, $u^{\delta}$. 

Let us now define $s^{\epsilon,\delta}$ as $s^{\epsilon,\delta}:=s^{\epsilon,\delta}_1+s^{\epsilon,\delta}_2$.

\begin{prop}\label{convergenciasolu}
If $2<p_i^*$, then $u^{\epsilon,\delta}_i(0)$ converges to $u^{0,\delta}_i(0)(=u^{\delta}_i)$ strongly in $W^{1,p_i}(D^{\delta}_i)$ and $s^{\epsilon,\delta}_i(0)$ converges to $s^{0,\delta}_i(0)(=s^{\delta}_i(0))$ when $\epsilon \to 0$.
\end{prop}

\begin{proof}
Let us begin by observing that
$$
\J^{\epsilon,\delta}_{0,i}(v):=\int_{D^{\delta}_i}\frac{\big(|\nabla v |^2 
 + \epsilon^2\big)^{\frac{p_i}{2}}}{p_i}\, dy +\int_{D^{\delta}_i} \frac{|v|^2}{2} \, dy -\int_{D^{\delta}_i} v f^{\epsilon}\, dy.
$$
Now let us denote
$$
\J^{\delta}_i(v):=\int_{D^{\delta}_i}\frac{|\nabla v |^2}{p_i}\, dy +\int_{D^{\delta}_i} \frac{|v|^2}{2} \, dy -\int_{D^{\delta}_i} v f\, dy.
$$
Observe that $\J^{\epsilon,\delta}_{0,i}$, $\J^{\delta}_i(v):X^{\delta}_i \to \R$. By Theorem \ref{convergencia sobre acotados}, it is enough to prove that $\J^{\epsilon,\delta}_{0,i}$ $\Gamma$- converges to $\J^{\delta}_i$ in $W^{1,p_i}(D^{\delta}_i)$ for the weak topology.

First, let $v^{\epsilon}\rightharpoonup v$ weakly in $W^{1,p_i}(D^{\delta}_i)$. Let us observe that $v \in X^{\delta}_i$ since $X^{\delta}_i$ is weakly closed. 

Observe that the first and second terms in $\J^{\delta}_{i}$ are convex and strongly continuous, therefore weakly lower semicontinuous. And the third term is linear and continuous, therefore weakly continuous.

Therefore,
$$
\J^{\delta}_{i}(v)\leq \lim \J^{\delta}_{i}(v^{\epsilon})\leq \liminf\J^{\epsilon,\delta}_{0,i}(v^{\epsilon})+\int_{D^{\delta}_i} v^{\epsilon}(f-f^{\epsilon}).
$$ 
Applying H\"older's inequality for the last term above, we have that
$$
\int_{D^{\delta}_i} v^{\epsilon}(f-f^{\epsilon})\leq ||v^{\epsilon}||_{p_i}||f-f^{\epsilon}||_{p'_i}.
$$
Since $||v^{\epsilon}||_{p_i}$ is bounded (because of the weak convergence) and $f^{\epsilon} \rightarrow f$ in $L^{p'_i}$, the last term goes to 0.

Now, taking $\{v^{\epsilon}\}=v$ as recovery sequence, we have that $\J^{\epsilon,\delta}_{0,i}(v)\rightarrow \J^{\epsilon,\delta}(v)$, which completes the proof.
\end{proof}

Performing analogous computations as in the previous section, we can see that $s^{\epsilon,\delta}_i(t)$ is differentiable at $t=0$ and
$$
\frac{ds^{\epsilon,\delta}_i(0)}{dt}=R^{\epsilon,\delta}_{i}(u^{\epsilon, \delta}_i) - \int_{D^{\delta}_i} u^{\epsilon, \delta}_i f \div V\, dy - \int_{D^{\delta}_i} u^{\epsilon, \delta}_i \nabla f^{\epsilon}\cdot V\, dy
$$
where
$$
R^{\epsilon,\delta}_{i}(v):=\int_{D^{\delta}_i}\frac{(|\nabla v|^{2}+\epsilon^{2})^\frac{p_i}{2}}{p_i}\div V-(|\nabla v|^{2}+\epsilon^{2})^{\frac{p_i}{2}-1} \nabla v\cdot\nabla v DV+ \div V\frac{|v|^2}{2}\, dy.
$$

Since the expression of $\frac{ds^{\epsilon,\delta}_{i}(0)}{dt}$ given above only involves first derivatives, we can conclude the following result from Corollary \ref{convergenciasolu}.  
\begin{prop}\label{convergenciaderivada}
$\frac{ds^{\epsilon,\delta}_{i}(0)}{dt}$ converges to $\frac{ds^{0,\delta}_{i}(0)}{dt}$ when $\epsilon \to 0$.
\end{prop}

Observe that, by Propositions \ref{convergenciasolu} and \ref{convergenciaderivada}, if we find an expression for the shape derivative of the regularized operator, we will have found one for the original operator.

\subsection{Improvement of the formula for the regularized operator.}
Our main concern in this part of our work will be to find a formula for the shape derivative that does not involve second order derivatives. Therefore, we will be able to pass to the limit when $\epsilon$ goes to $0$. And so, by Proposition \ref{convergenciaderivada}, we will have found an expression for the shape derivative of the original operator.

We start with some preliminaries computations in which we will see the need to have $C^{2}$ regularity for our solutions.
Since
\begin{align*}
\div((|\nabla u^{\epsilon}|^{2}+\epsilon^{2})^{\frac{p_i}{2}}\cdot V)&=\frac{p_i}{2}(|\nabla u^{\epsilon}|^{2}+\epsilon^{2})^{\frac{p_i}{2}-1} D(|\nabla u^{\epsilon}|^{2}+\epsilon^{2}) \cdot V+(|\nabla u^{\epsilon}|^{2}+\epsilon^{2})^{\frac{p_i}{2}}\div V\\
&=p_i(|\nabla u^{\epsilon}|^{2}+\epsilon^{2})^{\frac{p_i}{2}-1} \nabla u^{\epsilon} D^{2}u^{\epsilon} \cdot V+(|\nabla u^{\epsilon}|^{2}+\epsilon^{2})^{\frac{p_i}{2}}\div V,
\end{align*}
we have that
$$
\frac{1}{p_i}\int_{D^{\delta}_i}(|\nabla u^{\epsilon}|^{2}+\epsilon^{2})^{\frac{p_i}{2}}\div V=\frac{1}{p_i}\int_{D^{\delta}_i}\div((|\nabla u^{\epsilon}|^{2}+\epsilon^{2})^{\frac{p_i}{2}}V)- \int_{D^{\delta}_i}(|\nabla u^{\epsilon}|^{2}+\epsilon^{2})^{\frac{p_i}{2}-1} \nabla u^{\epsilon} D^{2}u^{\epsilon} \cdot V
$$
Therefore,
$$
\frac{ds^{\epsilon,\delta}_i}{dt}(0)=\frac{1}{p_i}\int_{D^{\delta}_i}\div((|\nabla u^{\epsilon}|^{2}+\epsilon^{2})^{\frac{p_i}{2}}V)- \int_{D^{\delta}_i}(|\nabla u^{\epsilon}|^{2}+\epsilon^{2})^{\frac{p_i}{2}-1} \nabla u^{\epsilon} D^{2}u^{\epsilon} \cdot V
$$
$$
-\int_{D^{\delta}_i}(|\nabla u^{\epsilon}|^{2}+\epsilon^{2})^{\frac{p_i}{2}-1}  \nabla u^{\epsilon} \cdot \nabla u^{\epsilon} DV\, dy+\frac{1}{2}\int_{D^{\delta}_i} \div (|u^{\epsilon}|^2V)\, dy
$$
$$
-\int_{D^{\delta}_i}u^{\epsilon}\nabla u^{\epsilon} \cdot V \, dy-\int_{D^{\delta}_i} u^{\epsilon} f^{\epsilon} \div V\, dy - \int_{D^{\delta}_i} u^{\epsilon} \nabla f^{\epsilon}\cdot V\, dy.
$$
Let us call $\nu^{\delta}_i$ the exterior unit normal vector to $\partial D^{\delta}_i$ and observe that, since $\sop V \subset\subset \Omega$,
$$
\int_{D^{\delta}_i} \div (|u^{\epsilon}|^2V)\, dy=\int_{\Gamma^{\delta}_i} |u^{\epsilon}|^{2}V \cdot \nu^{\delta}_i \, dS.
$$
Since $u^{\epsilon}$ is a weak solution of our equation, for every test function $\varphi$ we have
$$
\int_{D^{\delta}_i}(|\nabla u^{\epsilon}|^{2}+{\epsilon}^2)^{\frac{p_i-2}{2}}\nabla u^{\epsilon}\nabla \varphi+\int_{D^{\delta}_i} u^{\epsilon} \varphi=\int_{D^{\delta}_i} f^{\epsilon} \varphi.
$$
Let us consider $\varphi=\nabla u^{\epsilon} \cdot V$ as a test function. Since $\nabla (\nabla u^{\epsilon} \cdot V)=D^2u^{\epsilon} \cdot V^t+\nabla u^{\epsilon} DV$, we get
$$
\int_{D^{\delta}_i}(|\nabla u^{\epsilon}|^{2}+{\epsilon}^2)^{\frac{p_i-2}{2}}\nabla u^{\epsilon}(D^2u^{\epsilon} \cdot V^t+\nabla u^{\epsilon} DV)=\int_{\Gamma^{\delta}_i}(|\nabla u^{\epsilon}|^{2} +{\epsilon}^2)^{\frac{p_i-2}{2}}\nabla u^{\epsilon} \cdot \eta \nabla u^{\epsilon} \cdot V+\int_{D^{\delta}_i} (f^{\epsilon}-u^{\epsilon}) \nabla u^{\epsilon} \cdot V.
$$
And, since $V$ has compact support in $\Omega$, we arrive at
$$
\int_{\partial D^{\delta}_i}(|\nabla u^{\epsilon}|^{2} +{\epsilon}^2)^{\frac{p_i-2}{2}}\nabla u^{\epsilon} \cdot \eta \nabla u^{\epsilon} \cdot V=\int_{\Gamma^{\delta}_i}(|\nabla u^{\epsilon}|^{2} +{\epsilon}^2)^{\frac{p_i-2}{2}}\nabla u^{\epsilon} \cdot \nu_i^{\delta} \nabla u^{\epsilon} \cdot V.
$$

Therefore, taking into account that $\nabla u^{\epsilon} D^{2} u^{\epsilon} \cdot V =\nabla u^{\epsilon} \cdot D^{2} u^{\epsilon} V^T$, we have that
$$
\frac{ds^{\epsilon, \delta}_i}{dt}(0)= \frac{1}{p_i} \int_{\Gamma_i^\delta}\Big(|\nabla u^{\epsilon}|^2 + \epsilon^2\Big)^\frac{p_i}{2}V\nu_i^\delta- \int_{\Gamma_i^\delta}\Big(|\nabla u^{\epsilon}|^2 + \epsilon^2\Big)^{\frac{p_i}{2}-1} \nabla u^{\epsilon} \nu_i^\delta \nabla u^{\epsilon} V
$$
$$
+\frac{1}{2}\int_{\Gamma_i^\delta}|u^{\epsilon}|^2V \nu_i^\delta- \int_{D_i^\delta} \Big(\underbrace{f^{\epsilon}\nabla u^{\epsilon} \cdot V + u^{\epsilon} f^{\epsilon} \div V + u^{\epsilon}\nabla f^{\epsilon} \cdot V}_{\div(u^{\epsilon}f^{\epsilon}V)}\Big)\, dy.
$$
Again since V has compact support in $\Omega$, we have that
$$
\int_{D^{\delta}_i}\div (u^{\epsilon}f^{\epsilon}V)\, dy=\int_{\partial D^{\delta}_i}u^{\epsilon}f^{\epsilon}V \cdot \nu^{\delta}_i \, dS=\int_{\Gamma^{\delta}_i}u^{\epsilon}f^{\epsilon}V \cdot \nu^{\delta}_i \, dS.
$$
Observe that we arrive at an expression for the shape derivative that does not involve second order derivatives of $u^{\epsilon}$:
$$
\frac{ds^{\epsilon, \delta}_i}{dt}(0)= \frac{1}{p_i} \int_{\Gamma_i^\delta}\Big(|\nabla u^{\epsilon}|^2 + \epsilon^2\Big)^\frac{p_i}{2}V\nu_i^\delta- \int_{\Gamma_i^\delta}\Big(|\nabla u^{\epsilon}|^2 + \epsilon^2\Big)^{\frac{p_i}{2}-1} \nabla u^{\epsilon} \nu_i^\delta \nabla u^{\epsilon} V
$$
$$
+\frac{1}{2}\int_{\Gamma_i^\delta}|u^{\epsilon}|^2V \nu_i^\delta-\int_{\Gamma^{\delta}_i}u^{\epsilon}f^{\epsilon}V \cdot \nu^{\delta}_i \, dS. 
$$

\subsection{Back to the original operator: the limit when $\epsilon$ goes to $0$.}
Now we able to apply Tolksdorf's regularity estimates (see \cite{Tolksdorf}). These estimates give us uniform bounds for $||u^{\epsilon}||_{C^{1,\alpha}}$ so we have $u^{\epsilon}\to u$ in $C^{1}$. And so we can pass to the limit when $\epsilon$ goes to $0$. Therefore,
$$
\frac{ds^{0, \delta}_i}{dt}(0)= \frac{1}{p_i} \int_{\Gamma_i^\delta}|\nabla u|^{p_i}V\nu_i^\delta- \int_{\Gamma_i^\delta}|\nabla u|^{p_i-2} \nabla u \nu_i^\delta \nabla u V+\frac{1}{2}\int_{\Gamma_i^\delta}|u|^2V \nu_i^\delta- \int_{\Gamma^{\delta}_i}ufV \cdot \nu^{\delta}_i \, dS.
$$
In conclusion we arrive at
$$
\frac{ds^{0,\delta}}{dt}(0)=\frac{ds^{0,\delta}_1}{dt}(0)+\frac{ds^{0,\delta}_2}{dt}(0)
$$
$$
=\frac{1}{p_1} \int_{\Gamma_1^\delta}|\nabla u|^{p_1}V\nu_1^\delta+\frac{1}{p_2} \int_{\Gamma_2^\delta}|\nabla u|^{p_2}V\nu_2^\delta- \int_{\Gamma_1^\delta}|\nabla u|^{p_1-2} \nabla u \nu_1^\delta \nabla u V- \int_{\Gamma_2^\delta}|\nabla u|^{p_2-2} \nabla u \nu_2^\delta \nabla u V
$$
$$
+\frac{1}{2}\int_{\Gamma_1^\delta}|u|^2V \nu_1^\delta- \int_{\Gamma^{\delta}_1}ufV \cdot \nu^{\delta}_1 \, dS+\frac{1}{2}\int_{\Gamma_2^\delta}|u|^2V \nu_2^\delta- \int_{\Gamma^{\delta}_2}ufV \cdot \nu^{\delta}_2 \, dS.
$$
Let us now observe that $\nu_1^\delta \rightarrow \nu_1$ and $\nu_2^\delta \rightarrow \nu_2=-\nu_1$ when $\delta \to 0$. Therefore, taking limit when $\delta \to 0$, the last four terms in the expression above vanish and so we have proved the following.

\begin{teo}
Let $\Omega\subset \R^N$ be open and bounded. Let $D_1, D_2\subset \Omega$ be such that \eqref{prop.particion} is satisfied, let $p=p_1\chi_{D_1} + p_2\chi_{D_2}$, where $1<p_1<p_2$ and $\Gamma=\bar{D_1}\cap\bar{D_2}$.

Let $V\colon \R^N\to\R^N$ be a Lipschitz deformation field, such that spt$(V)\subset\subset\Omega$ and let $s(t)$ be defined by \eqref{st}. Then, the following formula for the derivative $s'(0)$ holds:
$$
\frac{ds}{dt}(0)= \int_{\Gamma} \Big[\frac{|\nabla u|^{p}}{p}\Big]V \cdot \nu \, dS- \int_{\Gamma}[|\nabla u|^{p-2}] (\nabla u\cdot \nu) (\nabla u\cdot V)\, dS,
$$
where
$$
\int_\Gamma [f] G\cdot \nu\, dS := \lim_{\delta\to 0} \left(\int_{\Gamma_1^\delta} f G\cdot\nu_1\, dS - \int_{\Gamma_2^\delta} f G\cdot \nu_2\, dS\right). 
$$
\end{teo}

\appendix 

\section{Gamma convergence results}

In this appendix we will recall some basic concepts of $\Gamma-$convergence that are needed in the present paper. Although these results are well-known, we decide to include this appendix in order to make the paper self contained. Also, the results presented here are not stated in the most general form, but in a for that will be enough for our work. For a complete presentation of the theory of $\Gamma-$convergence, see the book of Dal Maso \cite{DalMaso}.

Let $\psi_n$ and $\psi$ defined in a topological space $X_\tau$ with $T^{2}$ topology. For our applications, $X_\tau$ will be a Banach space and we will consider the weak topology. Then, a family of functionals $\psi_n$ $\Gamma$-converges to $\psi$ if
\begin{itemize}
	\item (liminf inequality) $x_n \to_{\tau} x$ implies that $\psi(x)\leq \liminf \psi_n(x_n)$
and
  \item (limsup inequality) there exists $y_n \rightharpoonup x$ such that $\psi(x)\geq \limsup \psi_n(y_n)$.
\end{itemize}

\begin{teo}\label{convergencia sobre acotados}
Let $X$ be a Banach space, $C \subset X$ closed and convex. Let $\psi_n, \psi\colon C\to [-\infty,\infty]$ be weakly lower semicontinuous, strictly convex and uniformly coercive functionals (i.e. for every $\lambda$, the set $\{x \in C : \psi_n(x)\leq \lambda\}\subset B_r$ for every $n$), then $\inf_{C}\psi_n=\min_{C}\psi_n \rightarrow \inf_{C}\psi=\min_{C}\psi$.

And, if $x_n \in C$ is such that $\psi_n(x_n)=\min_{C}\psi_n$, then $(x_n)$ es precompact and $\psi(x_0)=\min_{C}\psi$ where $x_0=\lim x_n$.

\end{teo}

\begin{proof}
Let us start by observing that, since $\psi_n$ weakly lower semicontinuous, strictly convex and uniformly coercive functionals, for every $n$ there is a unique $x_n$ such that $\psi_n(x_n)=\inf_{C}\psi_n$ and $(x_n)$ is bounded if $\psi_n(x_n)$ is bounded.
Let us consider now the following recovery function: $x \in C$ such that $y_n \rightharpoonup x$. Therefore,
$$
\psi_n(x_n)=\inf_{C}\psi_n\leq \psi_n(y_n).
$$
And so for every $x$ we have that
$$
\limsup \psi_n(x_n)\leq \limsup \psi_n(y_n)\leq \psi(x).
$$
Therefore,
$$
\limsup \psi_n(x_n)\leq \inf_{C}\psi<\infty
$$
and we can conclude that $x_n \in \{x \in C : \psi_n(x)\leq \lambda\}\subset B_r$ for every $n\geq n_0$ taking $\lambda=\inf_{C}\psi+1$. So $(x_n)$ is bounded and, via subsequences if necessary, $x_n\rightharpoonup x_0 \in C$ (remember that $C$ is convex and closed, therefore weakly closed)).

Finally, observing that
$$
\inf_{C}\psi\leq \psi(x_0)\leq \liminf \psi_n(x_n)\leq \liminf (\inf_{C}\psi_n), 
$$
the proof is completed.
\end{proof}

\begin{remark}\label{obs gamma convergencia}
If $\psi_n \rightarrow \psi$ point-wise, the inequality of the inferior limit (it is enough to take $y_n$ equal to $x$ for every $n$) always holds. Therefore, to obtain the convergence of the functionals it would only be necessary to check the superior limit inequality.  
\end{remark}

\section*{Acknowledgements}

This paper is partially supported by grants UBACyT 20020130100283BA, CONICET PIP 11220150100032CO and ANPCyT PICT 2012-0153. 

C. Baroncini is a doctoral fellow of CONICET and J. Fern\'andez Bonder is a member of CONICET.

\def\ocirc#1{\ifmmode\setbox0=\hbox{$#1$}\dimen0=\ht0 \advance\dimen0
  by1pt\rlap{\hbox to\wd0{\hss\raise\dimen0
  \hbox{\hskip.2em$\scriptscriptstyle\circ$}\hss}}#1\else {\accent"17 #1}\fi}
  \def\ocirc#1{\ifmmode\setbox0=\hbox{$#1$}\dimen0=\ht0 \advance\dimen0
  by1pt\rlap{\hbox to\wd0{\hss\raise\dimen0
  \hbox{\hskip.2em$\scriptscriptstyle\circ$}\hss}}#1\else {\accent"17 #1}\fi}
\providecommand{\bysame}{\leavevmode\hbox to3em{\hrulefill}\thinspace}
\providecommand{\MR}{\relax\ifhmode\unskip\space\fi MR }
\providecommand{\MRhref}[2]{%
  \href{http://www.ams.org/mathscinet-getitem?mr=#1}{#2}
}
\providecommand{\href}[2]{#2}
\bibliographystyle{plain}
\bibliography{biblio}

\end{document}